\documentclass[11pt,a4paper,reqno]{amsart}
\usepackage{clrscode}
\usepackage{mathrsfs,amssymb}

\newcommand{\scrA}{\mathscr{A}}
\newcommand{\scrR}{\mathscr{R}}
\newcommand{\ol}[1]{\overline{#1}}

\newtheorem{theorem}{Theorem}

\newtheorem{corollary}{Corollary}

\newtheorem{proposition}{Proposition}
\newtheorem{lemma}{Lemma}

\begin{document}
\begin{center}

   \title[A Note on Small Overlap Monoids]{A Note on the Definition of \\ Small Overlap Monoids}
    \keywords{monoid, semigroup, word problem, finite presentation, small overlap, small cancellation}
    \subjclass[2000]{20M05}

   \maketitle

    Mark Kambites \\

    \medskip

    School of Mathematics, \ University of Manchester, \\
    Manchester M13 9PL, \ England.

    \medskip

    \texttt{Mark.Kambites@manchester.ac.uk} \\

    \medskip

\end{center}

\begin{abstract}
Small overlap conditions are simple and natural combinatorial conditions
on semigroup and monoid presentations, which serve to limit the complexity
of derivation sequences between equivalent words in the generators. They
were introduced by J.~H.~Remmers, and more recently have been extensively
studied by the present author. However, the definition of small overlap
conditions hitherto used by the author was slightly more restrictive than
that introduced by Remmers; this note eliminates this discrepancy by extending
the recent methods and results of the author to apply to Remmers' small overlap monoids in full generality.
\end{abstract}

Small overlap conditions are simple and natural combinatorial conditions
on semigroup and monoid presentations, which serve to limit the complexity
of derivation sequences between equivalent words in the generators. Introduced
by J.~H.~Remmers \cite{Higgins92,Remmers71,Remmers80}, and more recently
studied by the present author \cite{K_generic,K_smallover1,K_smallover2},
they are the natural semigroup-theoretic analogue of the small cancellation
conditions widely used in combinatorial group theory \cite{Lyndon77}.

The definitions of small overlap conditions originally introduced by
Remmers are slightly more general than those used by the present author.
The aims of this note are to clarify this distinction, and then to extend
the methods and results introduced in \cite{K_smallover1,K_smallover2} to
the full generality of small overlap monoids as studied by Remmers.

In addition to this introduction, this article comprises three sections.
In Section~\ref{sec_prelim} we briefly recall the definitions of small
overlap conditions, and also discuss the distinction between Remmers' and
the author's definitions. In Section~\ref{sec_main} we show how to extend
the key technical results from \cite{K_smallover1}, from the slightly
restricted setting considered there to Remmers' small overlap conditions
in their more general form. Finally, Section~\ref{sec_apps} applies the
results of the previous section to extend the main results of
\cite{K_smallover1,K_smallover2} to the more general case.

The proofs for certain of the results in this paper are very similar (in
some cases identical) to arguments used in previous papers \cite{K_smallover1,K_smallover2}. In the interests of brevity we refrain
from repeating these, instead providing detailed references. Hence, while the
results of this paper may be read in isolation, the reader wishing to fully
understand the proofs is advised to read it in conjunction with
\cite{K_smallover1,K_smallover2}.

\section{Small Overlap Monoids}\label{sec_prelim}

We assume familiarity with basic notions of combinatorial semigroup
theory, including free semigroups and monoids, and semigroup and monoid
presentations. Except where stated otherwise, we assume we have a fixed
finite presentation for a monoid (or semigroup, the difference being
unimportant). Words are assumed to be drawn from the free monoid on the
generating alphabet unless otherwise stated. We write $u = v$ to indicate that two words are
equal in the free monoid or semigroup, and $u \equiv v$ to indicate that they represent
the same element of the monoid or semigroup presented. We say that a word $p$ is a
\textit{possible prefix} of $u$ if there exists a (possibly empty) word
$w$ with $pw \equiv u$, that is, if the element represented by $u$ lies in
the right ideal generated by the element represented by $p$. The empty
word is denoted $\epsilon$.

A \textit{relation word} is a word which occurs as one side of a
relation in the presentation. A \textit{piece} is a word in the
generators which occurs as a factor in sides of two \textit{distinct} relation
words, or in two different (possibly overlapping) places within one
side of a relation word. Note that this definition differs slightly from
that used in \cite{K_smallover1,K_smallover2} in the presence of the word
``distinct''; we shall discuss the
significance of this shortly. By convention, the empty word is always a piece.
We say that a presentation is \textit{weakly $C(n)$}, where
$n$ is a positive integer, if no relation word can be written as the product
of \textit{strictly fewer than} $n$ pieces. Thus for each $n$, being weakly
$C(n+1)$ is a stronger condition than being weakly $C(n)$.

In \cite{K_smallover1,K_smallover2} we used a slightly
more general definition of a piece, following through with which led
to slightly more
restrictive conditions $C(n)$; the author is grateful to Uri Weiss for
pointing out this discrepancy.
Specifically, in \cite{K_smallover1,K_smallover2} we defined a piece to be a
word which
occurs more than once as a factor of words in the \textit{sequence} of
relation words. Under this definition, if the same relation word appears
twice in a presentation then it is considered to be a piece, and so the
presentation fails to satisfy $C(2)$. By contrast, Remmers defined a piece
to be a word which appears more than once as a factor of words in the
\textit{set} of relation words. The effect of this is that Remmers' definition
permits $C(2)$ (and higher) presentations to have relations of, for example,
the form $(u, v_1)$ and $(u, v_2)$ with $v_1 \neq v_2$. (Equivalently, one
could choose to define a piece in terms of the sequence of relation words
but permit ``$n$-ary'' relations of the form $(u,v_1,v_2)$, to be interpreted
as equivalent to relations $(u,v_1)$ and $(u,v_2)$). In this paper, we say
that a presentation is \textit{strongly} $C(n)$ if it is weakly $C(n)$ and
has no repeated relation words, that is, if it satisfies the condition which
was called \textit{C(n)} in \cite{K_smallover1,K_smallover2}.

In fact it transpires that the weakly $C(n)$ conditions still suffice
to establish the main methods and results of \cite{K_smallover1,K_smallover2}. However,
this fact is rather obscured by the technical details and notation in
\cite{K_smallover1,K_smallover2}. In particular, for a relation word $R$ we
defined $\ol{R}$ to be the (necessarily unique) word such that $R = \ol{R}$ or $\ol{R} = R$ is a relation in the
presentation. The extensive use of this notation makes it difficult
to convince oneself that the arguments in \cite{K_smallover1,K_smallover2}
do indeed apply in the more general case, so the aim of this paper is to
provide full proofs of the results of those papers in the more general setting.

For each relation word $R$, let $X_R$ and $Z_R$ denote respectively the
longest prefix of $R$ which is a piece, and the longest suffix of $R$
which is a piece. If the presentation is weakly $C(3)$ then $R$ cannot be
written as a product of two pieces, so this prefix and suffix cannot meet;
thus, $R$ admits a factorisation $X_R Y_R Z_R$ for some non-empty word $Y_R$.
If moreover the presentation is weakly $C(4)$, then the relation word $R$
cannot be written as a product of three pieces, so $Y_R$ is not a piece. The
converse also holds: a weakly $C(3)$ presentation such that no $Y_R$ is a piece 
is a weakly $C(4)$ presentation. We
call $X_R$, $Y_R$ and $Z_R$ the \textit{maximal piece prefix}, the
\textit{middle word} and the \textit{maximal piece suffix} respectively
of $R$.

Assuming now that the presentation is weakly $C(3)$,
we shall use the letters $X$, $Y$ and $Z$ (sometimes with adornments or
subscripts) exclusively to represent maximal piece prefixes, middle words
and maximal piece suffixes respectively of relation words; two such letters
with the same subscript or adornment (or with none) will be assumed to
stand for the appropriate factors of the same relation word.

We say that a relation word $\ol{R}$ is a \textit{complement} of a relation
$R$ if there are relation words $R = R_1, R_2, \dots, R_n = \ol{R}$ such
that either $(R_i, R_{i+1})$ or $(R_{i+1}, R_i)$ is a relation in the
presentation for $1 \leq i < n$. We say that $\ol{R}$ is a \textit{proper}
complement of $R$ if, in addition, $\ol{R} \neq R$. Abusing notation and
terminology slightly, if $R = X_R Y_R Z_R$ and $\ol{R} = X_{\ol{R}} Y_{\ol{R}} Z_{\ol{R}}$ then we
write $\ol{X_R} = X_{\ol{R}}$, $\ol{X_R Y_R} = X_{\ol{R}} Y_{\ol{R}}$ and so
forth. We say that $\ol{X_R}$ is a complement of $X_R$, and $\ol{X_R Y_R}$
is a complement of $X_R Y_R$.

A \textit{relation prefix} of a word
is a prefix which admits a (necessarily unique, as a consequence of the
small overlap condition) factorisation of the form $a X Y$ where $X$ and $Y$
are the maximal piece prefix and middle word respectively of some relation
word $XYZ$. An \textit{overlap prefix (of length $n$)} of
a word $u$ is a relation prefix which admits an (again necessarily unique)
factorisation of the form $b X_1 Y_1' X_2 Y_2' \dots X_n Y_n$ where
\begin{itemize}
\item $n \geq 1$;
\item $b X_1 Y_1' X_2 Y_2' \dots X_n Y_n$ has no factor of the form $X_0Y_0$,
where $X_0$ and $Y_0$ are the maximal piece prefix and middle word respectively
of some relation word, beginning before the end of the prefix $b$;
\item for each $1 \leq i \leq n$, $R_i = X_i Y_i Z_i$ is a relation word with
$X_i$ and $Z_i$ the maximal piece prefix and suffix respectively; and
\item for each $1 \leq i < n$, $Y_i'$ is a proper, non-empty prefix of $Y_i$.
\end{itemize}
Notice that if a word has a relation prefix, then the shortest such must
be an overlap prefix. A relation prefix $a XY$ of a word $u$ is called
 \textit{clean} if $u$ does \textbf{not} have a prefix
$$a XY' X_1 Y_1$$
where $X_1$ and $Y_1$ are the maximal piece prefix and middle word respectively
of some relation word, and $Y'$ is a proper, non-empty prefix of $Y$. As in
\cite{K_smallover1}, clean overlap prefixes will play a crucial role in what
follows.

If $u$ is a word and $p$ is a piece, we say that $u$ is \textit{$p$-active} if $p u$ has a relation prefix
$aXY$ with $|a| < |p|$, and \textit{$p$-inactive} otherwise.

\section{Technical Results}\label{sec_main}

In this section we show how some technical results and methods from
\cite{K_smallover1} concerning strongly $C(4)$ monoids
can be extended to cover weakly $C(4)$ monoids. We assume
throughout initially a fixed monoid presentation which is weakly $C(4)$.
The following three
foundational statements are completely unaffected by our revised definitions,
and can still be proved exactly as in \cite{K_smallover1}.

\begin{proposition}\label{prop_overlapprefixnorel}
Let $a X_1 Y_1' X_2 Y_2' \dots X_n Y_n$
be an overlap prefix of some word. Then this prefix
contains no relation word as a factor, except possibly the suffix $X_n Y_n$ in
the case that $Z_n = \epsilon$.
\end{proposition}

\begin{proposition}\label{prop_opgivesmop}
Let $u$ be a word. Every overlap prefix of $u$ is contained in a
clean overlap prefix of $u$.
\end{proposition}

\begin{corollary}\label{cor_nomopnorel}
If a word $u$ has no clean overlap prefix, then it contains
no relation word as a factor, and so if $u \equiv v$ then $u = v$.
\end{corollary}

The following lemma is essentially a restatement of \cite[Lemma~1]{K_smallover1}
using our new notation. The proof is essentially the same as in
\cite{K_smallover1}, with the addition of an obvious inductive argument to
allow for the fact that several rewrites may be needed to obtain $\ol{XYZ}$
from $XYZ$.

\begin{lemma}\label{lemma_staysclean}
Suppose $u = w XYZ u'$ with $w XY$ a clean overlap prefix and
$\ol{XYZ}$ is a complement of $XYZ$. Then $w \ol{XY}$ is a clean
overlap prefix of $w \ol{XYZ} u'$.
\end{lemma}

From now on, we shall assume that our presentation is weakly $C(4)$.
We are now ready to prove our first main technical result, which is an
analogue of \cite[Lemma 2]{K_smallover1}, and is fundamental to our
approach to weakly $C(4)$ monoids.

\begin{lemma}\label{lemma_overlap}
Suppose a word $u$ has clean overlap prefix $w X Y$. If
$u \equiv v$ then $v$ has overlap prefix $w \ol{X Y}$ for some
complement $\ol{XYZ}$ of $XYZ$, and
no relation word occurring as a factor of $v$ overlaps this prefix,
unless it is $\ol{X Y Z}$ in the obvious place.
\end{lemma}

\begin{proof}
Since $w X Y$ is an overlap prefix of $u$, it has by definition a
factorisation
$$w XY = a X_1 Y_1' \dots X_{n} Y_{n}' X Y$$
for some $n \geq 0$. We use this fact to prove the claim by induction on
the length $r$ of a rewrite sequence (using the defining relations) from
$u$ to $v$. 

In the case $r = 0$, we have $u = v$, so $v$ certainly has (clean) overlap
prefix $w XY$.
By Proposition~\ref{prop_overlapprefixnorel}, no relation word factor can
occur entirely within this prefix, unless it is the suffix $X Y$ and $Z = \epsilon$. If
a relation word factor of $v$ overlaps the end of the given overlap prefix
and entirely contains $XY$ then, since $XY$ is not a piece, that
relation word must clearly be $XYZ$. Finally,
a relation word cannot overlap the end of the given overlap prefix but
not contain the suffix $XY$, since this would clearly contradict either the
fact that the given overlap prefix is clean, or the fact that $Y$ is not a
piece.

Suppose now for induction that the lemma holds for all values less than $r$,
and that there is a rewrite sequence from $u$ to $v$ of length $r$. Let
$u_1$ be the second term in the sequence, so that $u_1$ is obtained from
$u$ by a single rewrite using the defining relations, and $v$ from $u_1$
by $r-1$ rewrites.

Consider the relation word in $u$ which is to be rewritten in order to
obtain $u_1$, and in
particular its position in $u$. By Proposition~\ref{prop_overlapprefixnorel},
this relation word cannot be contained in the clean overlap prefix $w XY$,
unless it is $X Y$ where $Z = \epsilon$.

Suppose first that the relation word to be rewritten contains the final
factor $Y$
of the given clean overlap prefix. (Note that this covers in particular the
case that the relation word is $XY$ and $Z = \epsilon$.)
From the weakly $C(4)$ assumption we know that $Y$ is not a piece, so we may deduce
that the relation word is $X Y Z$ contained in the obvious place. In
this case, applying the rewrite clearly leaves $u_1$ with a prefix 
$w \hat{X} \hat{Y}$ for some complement $\hat{X} \hat{Y} \hat{Z}$ of
$XYZ$. By Lemma~\ref{lemma_staysclean}, this is a clean overlap
prefix.  Now $v$ can be obtained from
$u_1$ by $r-1$ rewrite steps, so it follows from the inductive hypothesis
that $v$ has overlap prefix
$w \ol{XY}$ where $\ol{XYZ}$ is a complement of $\hat{X} \hat{Y} \hat{Z}$ and hence of
$XY$. It follows also that no relation word occurring as a factor of $v$
overlaps this prefix, unless it is $\ol{X Y Z}$; this
completes the proof in this case.

Next, we consider the case in which the relation word factor in $u$ to be
rewritten does not contain the final factor $Y$ of the clean overlap
prefix, but does overlap with the end of the clean overlap prefix. Then
$u$ has a factor of the form $\hat{X} \hat{Y}$, where $\hat{X}$ is the
maximal piece prefix and $\hat{Y}$ the middle word of a relation word,
which overlaps $X Y$, beginning after the start of $Y$. This clearly
contradicts the assumption that the overlap prefix is clean.

Finally, we consider the case in which the relation word factor in $u$
which is to be rewritten does not overlap the given clean overlap prefix
at all. Then obviously, the given clean overlap prefix of $u$ remains an
overlap prefix of $u_1$. If this overlap prefix is clean, then a simple
application of the inductive hypothesis again suffices to prove that $v$
has the required property.

There remains, then, only the case in which the given overlap prefix is
no longer clean in $u_1$. Then by definition there exist words $\hat{X}$ and
$\hat{Y}$, being a maximal piece prefix and middle word respectively of some relation
word, such
that $u_1$ has the prefix
$$a X_1 Y_1' \dots X_n Y_n' X Y' \hat{X} \hat{Y}$$
for some proper, non-empty prefix $Y'$ of $Y$.
Now certainly this is not a prefix of $u$, since this would contradict
the assumption that $a X_1 Y_1' \dots X_n Y_n' XY$
 is a clean overlap
prefix of $u$. So we deduce that $u_1$ can be transformed to $u$ by
rewriting a relation word
overlapping the final $\hat{X}\hat{Y}$. This relation word factor cannot contain
the entire of this factor $\hat{X}\hat{Y}$, since then it would overlap with
the prefix $a X_1 Y_1' \dots X_n Y_n X Y$, which would 
again contradict the assumption that this prefix is a clean overlap prefix of
$u$. Nor can the relation word contain the final factor $\hat{Y}$, since $\hat{Y}$ is not a piece.
Hence, $u_1$ must have a prefix
$$a X_1 Y_1' \dots X_{n-1} Y_{n-1}' X_n Y_n' X Y' \hat{X} \hat{Y}' R$$
for some relation word and proper, non-empty prefix $\hat{Y}'$ of $\hat{Y}$ and
some relation word $R$. Suppose $R = X_R Y_R Z_R$ where $X_R$ and $Z_R$
are the maximal piece prefix and suffix respectively. Then it is readily
verified that
$$a X_1 Y_1' \dots X_{n-1} Y_{n-1}' X_n Y_n' X Y' \hat{X} \hat{Y}' X_R Y_R$$
is a clean overlap prefix of $u_1$. Indeed, the fact it is an overlap prefix
is immediate, and if it were not clean then some factor of $u_1$ of the form
$\tilde{X} \tilde{Y}$ would have to overlap the end of the given prefix; but
this factor would either be contained in $Y_R Z_R$ (contradicting the
fact that $\tilde{X}$ is a maximum piece prefix of $\tilde{X} \tilde{Y} \tilde{Z}$)
or would contain a non-empty suffix of $Y_R$ followed by $Z_R$ (contradicting
the fact that $Z_R$ is a maximum piece prefix of $X_R Y_R Z_R$).

Now by the inductive
hypothesis, $v$ has prefix
\begin{equation}\label{vprefix2}
a X_1 Y_1' \dots X_{n-1} Y_{n-1}' X_n Y_n' X Y' \hat{X} \hat{Y}' \ol{X_R Y_R}.
\end{equation}
for some complement $\ol{X_R Y_R}$ of $X_R Y_R$. But now
$v$ has prefix
$$a X_1 Y_1' \dots X_{n-1} Y_{n-1}' X_n Y_n' X Y' \hat{X} \hat{Y}'$$
which in turn has prefix
\begin{equation}\label{vprefix3}
a X_1 Y_1' \dots X_{n-1} Y_{n-1}' X_n Y_n' X Y.
\end{equation}
Moreover, by Proposition~\ref{prop_overlapprefixnorel}, the prefix
\eqref{vprefix2} of $v$ contains no relation
word as a factor, unless it is the final factor $\ol{X_R Y_R}$
and $\ol{Z_R} = \epsilon$, and it follows
easily that no relation word factor overlaps the prefix \eqref{vprefix3}
of $v$.
\end{proof}

The following results are now proved exactly as their analogues in
\cite{K_smallover1}.

\begin{corollary}\label{cor_noncleanprefix}
Suppose a word $u$ has (not necessarily clean) overlap prefix
$w XY$. If $u \equiv v$ then $v$ has a
prefix $w$ and contains no relation word overlapping this prefix.
\end{corollary}

\begin{proposition}\label{prop_dumpprefix}
Suppose a word $u$ has an overlap prefix $a X Y$ and that
$u = a X Y u''$. Then $u \equiv v$ if and only if $v = a v'$ where
$v' \equiv X Y u''$.
\end{proposition}

\begin{proposition}\label{prop_inactive}
Let $u$ be a word and $p$ a piece.
If $u$ is $p$-inactive then $p u \equiv v$ if and only if $v = p w$
for some $w$ with $u \equiv w$.
\end{proposition}

\begin{proposition}\label{prop_coactive}
Let $p_1$ and $p_2$ be pieces and suppose $u$ is $p_1$-active and $p_2$-active.
Then $p_1$ and $p_2$ have a
common non-empty suffix, and if $z$ is their maximal common suffix then
\begin{itemize}
\item[(i)] $u$ is $z$-active;
\item[(ii)] $p_1 u \equiv v$ if and only if $v = z_1 v'$ where $z_1 z = p_1$ and
            $v' \equiv z u$; and
\item[(iii)] $p_2 u \equiv v$ if and only if $v = z_2 v'$ where $z_2 z = p_2$
            and $v' \equiv z u$.
\end{itemize}
\end{proposition}

\begin{corollary}\label{cor_actsame}
Let $p_1$ and $p_2$ be pieces. Suppose $p_1 u \equiv p_1 v$ and $u$
is $p_2$-active. Then $p_2 u \equiv p_2 v$.
\end{corollary}

The following is a strengthening of the \cite[Corollary 4]{K_smallover1}

\begin{corollary}\label{cor_eitheror}
Let $u$ and $v$ be words and $p_1, p_2, \dots, p_k$ be pieces.
Suppose there exist words $u = u_1, \dots, u_n = v$ such that
for $1 \leq i < n$ there exists $1 \leq j_i \leq k$ with
$p_{j_i} u_i \equiv p_{j_i} u_{i+1}$.
Then $p_j u \equiv p_j v$ for some $j$ with $1 \leq j \leq k$.
\end{corollary}
\begin{proof}
Fix $u$, $v$ and $p_1, \dots, p_k$, and suppose $n$ is minimal such 
that a sequence $u_1, \dots, u_n$ with the hypothesized properties exists.
Our aim is thus to show that $n \leq 2$. Suppose for a contradiction
that $n > 2$.

If $u_2$ was $p_{j_2}$-inactive then by Proposition~\ref{prop_inactive} we
would have $u_2 \equiv u_3$ so that $p_{j_1} u_1 \equiv p_{j_1} u_2 \equiv p_{j_1} u_3$
which clearly contradicts the minimality assumption on $n$.
Thus, $u_2$ is $p_{j_2}$-active.
But now since $p_{j_1} u_1 \equiv p_{j_1} u_2$, we apply
Corollary~\ref{cor_actsame} to see that
$p_{j_2} u_1 \equiv p_{j_2} u_2 \equiv p_{j_2} u_3$, which again
contradicts the minimality of $n$.
\end{proof}

We now present a lemma which gives a set of mutually exclusive combinatorial
conditions, the disjunction of which is necessary and sufficient for two words
of a certain form to represent the same element.

\begin{lemma}\label{lemma_eq}
Suppose $u = X Y u'$ where $XY$ is a clean overlap prefix of
$u$. Then $u \equiv v$ if and only if one of the following mutually
exclusive conditions holds:
\begin{itemize}
\item[(1)] $u = XYZ u''$ and $v = XYZ v''$ and $\ol{Z} u'' \equiv \ol{Z} v''$
for some complement $\ol{Z}$ of $Z$;
\item[(2)] $u = X Y u'$, $v = X Y v'$, and $Z$ fails to be a
prefix of at least one of $u'$ and $v'$, and $u' \equiv v'$;
\item[(3)] $u = X Y Z u''$, $v = \ol{X} \ol{Y} \ol{Z} v''$ for some
uniquely determined proper complement $\ol{XYZ}$ of $XYZ$, 
and $\hat{Z} u'' \equiv \hat{Z} v''$ for some complement $\hat{Z}$
of $Z$;
\item[(4)] $u = X Y u'$, $v = \ol{X} \ol{Y} \ol{Z} v''$ for some uniquely
determined proper complement $\ol{XYZ}$ of $XYZ$ but
$Z$ is not a prefix of $u'$ and $u' \equiv Z v''$;
\item[(5)] $u = X Y Z u''$, $v = \ol{X} \ol{Y} v'$ for some uniquely determined
proper complement
$\ol{XYZ}$ of $XYZ$,
but $\ol{Z}$ is not a prefix of $v'$ and $\ol{Z} u'' \equiv v'$;
\item[(6)] $u = X Y u'$, $v = \ol{X} \ol{Y} v'$ for some uniquely determined proper complement
$\ol{XYZ}$ of $XYZ$, $Z$ is not
a prefix of $u'$ and $\ol{Z}$ is not a prefix of $v'$, but
$Z = z_1 z$, $\ol{Z} = z_2 z$, $u' = z_1 u''$, $v' = z_2 v''$ where
$u'' \equiv v''$ and $z$ is the maximal common suffix of $Z$ and $\ol{Z}$,
$z$ is non-empty, and $z$ is a possible prefix of $u''$.
\end{itemize}
\end{lemma}
\begin{proof}
It follows easily from the definitions that no complement of $XY$ is a
prefix of another. Hence, $v$ can have at most one of them as a prefix. Thus,
conditions (1)-(2) are not consistent with conditions (3)-(6), and the
prefixes of $v$ in (3)-(6) are uniquely determined. The mutual
exclusivity of (1) and (2) is self-evident from the definitions, and
likewise that of (3)-(6). 

It is easily verified that each of the conditions
(1)-(5) imply that $u \equiv v$. We show next that (6) implies that
$u \equiv v$. Since $z$ is a possible prefix of $u''$ and $u'' \equiv v''$,
we may write $u'' \equiv zx \equiv v''$ for some word $x$. Now we have
\begin{align*}
u = X Y u' = XY z_1 u'' &\equiv XY z_1 z x = XYZ x \\
&\equiv \ol{XYZ} x = \ol{XY} z_2 z x \equiv \ol{XY} z_2 v'' = \ol{XY} v' = v.
\end{align*}
It remains to show that $u \equiv v$ implies that one of
the conditions (1)-(6) holds. To this end, suppose $u \equiv v$;
then there is a rewrite sequence taking $u$ to $v$. 
By Lemma~\ref{lemma_overlap}, every term in this sequence will have prefix
which is a complement of $XY$, and this prefix can only be modified by
the application of a relation, both sides of which are complements of
$XYZ$, in the obvious place. We now prove the claim by case analysis.

By Lemma~\ref{lemma_overlap}, $v$ begins either with $XY$ or with some
proper complement $\ol{XY}$.
Consider first the case in which $v$ begins with $XY$; we split this into
two further cases depending on whether $u$ and $v$ both begin with the full
relation word $XYZ$; these will correspond respectively to conditions (1)
and (2) in the statement of the lemma.

\textbf{Case (1).} Suppose $u = XYZ u''$ and $v = X Y Z v''$.
Then clearly there is a rewrite sequence taking $u$ to $v$ which by
Lemma~\ref{lemma_overlap} can be
broken up as:
\begin{align*}
u &= XYZ u'' = X_0 Y_0 Z_0 u'' \to^* X_0 Y_0 Z_0 u_1 \to X_1Y_1Z_1 u_1 \to^* X_1 Y_1 Z_1 u_2 \\
&\to X_2 Y_2 Z_2 u_2 \to^* \dots \to X_n Y_n Z_n u_n \to^* X_n Y_n Z_n v'' = XYZ v'' = v
\end{align*}
where each prefix $X_i Y_i Z_i$ is a complement of $XYZ$, and 
none of the steps in the sequences indicated by $\to^*$ involves rewriting
a relation word overlapping with the prefix $X_i Y_i$.
It follows that there are rewrite sequences.
$$Z u'' \to^* Z u_1, \ Z_1 u_1 \to^* Z_1 u_2, \ Z_2 u_2 \to^* Z_2 u_3, \ \dots, \ Z_n u_n \to^* Z_n v''$$
Now by Corollary~\ref{cor_eitheror}, we have $Z_i u'' \equiv Z_i v''$ for
some $1 \leq i \leq n$, where $Z_i$ is a complement of $Z$ as required to
show that condition (1) holds.

\textbf{Case (2).} Suppose now that $u = X Y u'$, $v = XY v'$ and $Z$
fails to be a prefix of at least one of $u'$ and $v'$. We must show that
$u' \equiv v'$; suppose for a contradiction that this does not hold.
We again consider rewrite sequences
from $u = XY u'$ to $v = XY v'$. Again using Lemma~\ref{lemma_overlap}, we
see that there is either (i) such a sequence taking $u$ to $v$ containing
no rewrites of relation words overlapping the prefix $XY$, or (ii) such a
sequence taking $u$ to $v$ which can be broken up as:
\begin{align*}
u &= XY u' = X_0 Y_0 u'' \to^* X_0 Y_0 Z_0 u_1 \to X_1Y_1Z_1 u_1 \to^* X_1 Y_1 Z_1 u_2 \\
&\to X_2 Y_2 Z_2 u_2 \to^* \dots \to X_n Y_n Z_n u_n \to^* X_n Y_n Z_n v'' = X_n Y_n v' = XY v' = v
\end{align*}
where each prefix $X_i Y_i Z_i$ is a complement of $XYZ$, and 
none of the steps in the sequences indicated by $\to^*$ involves rewriting
a relation word overlapping with the prefix $X_i Y_i$.
In case (i) there is clearly a rewrite sequence
taking $u'$ to $v'$ so that $u' \equiv v'$ as required. In case (ii), there
are rewrite sequences.
$$u' \to^* Z u_1, \ Z_1 u_1 \to^* Z_1 u_2, \ Z_2 u_2 \to^* Z_2 u_3, \ \dots, \ Z_n u_n = Z u_n \to^* v'$$
Now if $u'$ does not begin with $Z$, we can deduce from
Proposition~\ref{prop_inactive} that $u_1$ is $Z$-active.
By Corollary~\ref{cor_eitheror}, we have $\hat{Z} u_1 \equiv \hat{Z} u_n$
for some complement $\hat{Z}$ of $Z$. Since $u_1$ is
$Z$-active, Corollary~\ref{cor_actsame} tells us that we also have
$Z u_1 \equiv Z u_n$. But now
$$u' \equiv Z u_1 \equiv Z u_n \equiv v'$$
so condition (2) holds. A similar argument applies if
$v'$ does not begin with $Z$.

\textbf{Case (3).} Suppose $u = XYZ u''$ and
$v = \ol{XYZ} v''$.
Then $u = XYZ u'' \equiv v \equiv XYZ v''$, so by the same argument as in case (1) we
have either $Zu'' \equiv Z v''$ or $\ol{Z} u'' \equiv \ol{Z} v''$ as required
to show that condition (3) holds.

\textbf{Case (4).} Suppose $u = XY u'$ and
 $v = \ol{XYZ} v''$ but $Z$ is not a prefix of $u'$. Then
$u = XY u' \equiv v \equiv XYZ v''$. Now applying the same argument as
in case (2) (with $XYZ v''$ in place of $v$ and setting $v' = Zv''$) we
have $u' \equiv v' = Z v''$ so that condition (4) holds.

\textbf{Case (5).} Suppose $u = XYZ u''$, $v = \ol{XY} v'$
but $\ol{Z}$ is not a prefix of $v'$. Then we have
$\ol{XYZ} u'' \equiv u \equiv v = \ol{XY} v'$, and moreover,
Lemma~\ref{lemma_staysclean} guarantees that $\ol{XY}$ is a clean overlap
prefix of $\ol{XYZ} u''$. Now applying the same
argument as in case (1) (but with $\ol{XYZ} u''$ in place of $u$ and
setting $u' = \ol{Z} u''$) we
obtain $u' \equiv v' = \ol{Z} u''$ so that condition (5) holds.

\textbf{Case (6).} Suppose $u = XY u'$, $v = \ol{XY} v'$ and that $Z$ is not a
prefix of $u'$ and $\ol{Z}$ is not a prefix of $v'$.
It follows this time that there is a rewrite sequence taking $u$ to $v$ of
the form
\begin{align*}
u = XY u' = & X_0 Y_0 u' \to^* X_0 Y_0 Z_0 u_1 \to X_1Y_1Z_1 u_1 \to^* X_1 Y_1 Z_1 u_2 \\
&\to X_2 Y_2 Z_2 u_2 \to^* \dots \to X_n Y_n Z_n u_n \to^* X_n Y_n v' = \ol{XY} v' = v
\end{align*}
where once more by Lemma~\ref{lemma_overlap}
each prefix $X_i Y_i Z_i$ is a complement of $XYZ$, and 
none of the steps in the sequences indicated by $\to^*$ involves rewriting
a relation word overlapping with the prefix $X_i Y_i$.
Now there are rewrite sequences.
$$u' \to^* Z u_1, \ Z_1 u_1 \to^* Z_1 u_2, \ Z_2 u_2 \to^* Z_2 u_3, \ \dots, \ Z_n u_n = \ol{Z} u_n \to^* v'$$
Notice that, since $u'$ does not begin with $Z$, we may deduce from
Proposition~\ref{prop_inactive} that $u_1$ is $Z$-active.
By Corollary~\ref{cor_eitheror}, we have $\hat{Z} u_1 \equiv \hat{Z} u_n$
for some complement $\hat{Z}$ of $Z$. Now since $u_1$ is
$Z$-active, Corollary~\ref{cor_actsame} tells us that we also have
$Z u_1 \equiv Z u_n$. But now
$$u' \equiv Z u_1 \equiv Z u_n$$ where $u'$ does not begin with $Z$, and
also $v' \equiv \ol{Z} u_n$ were $v'$ does not begin with $\ol{Z}$. By
applying Proposition~\ref{prop_inactive} twice, we deduce that $u_n$ is both
$Z$-active and $\ol{Z}$-active.

Let $z$ be the maximal common suffix of $Z$ and $\ol{Z}$. Then
applying Proposition~\ref{prop_coactive} (with $p_1 = Z$ and $p_2 = \ol{Z}$),
we see that $z$ is non-empty and
\begin{itemize}
\item $u' = z_1 u''$ where $Z = z_1 z$ and $u'' \equiv z u_n$; and
\item $v' = z_2 v''$ where $\ol{Z} = z_2 z$ and $v'' \equiv z u_n$.
\end{itemize}
But then we have
$u'' \equiv z u_n \equiv v''$ and also $z$ is a possible prefix of
$u''$ as required to show that condition (6) holds.
\end{proof}

\begin{lemma}\label{lemma_eqandprefix}
Suppose $u = X Y u'$ where $XY$ is a clean overlap prefix, and suppose
$p$ is a piece. Then $u \equiv v$ and $p$ is a possible prefix of $u$
if and only if one of the following mutually exclusive conditions holds:
\begin{itemize}
\item[(1')] $u = XYZ u''$ and $v = XYZ v''$ and
$\ol{Z} u'' \equiv \ol{Z} v''$ for some complement $\ol{Z}$ of $Z$, and
also $p$ is a prefix of some complement of $X$;

\item[(2')] $u = X Y u'$, $v = X Y v'$, and $Z$ fails to be a
prefix of at least one of $u'$ and $v'$, and $u' \equiv v'$,
and also either
\begin{itemize}
\item $p$ is a prefix of $X$; or
\item $p$ is a prefix of some complement of $X$ and $Z$ is a possible prefix of $u'$.
\end{itemize}

\item[(3')] $u = X Y Z u''$, $v = \ol{X} \ol{Y} \ol{Z} v''$ for some
uniquely determined proper complement $\ol{XYZ}$ of $XYZ$, and $\hat{Z} u'' \equiv \hat{Z} v''$
for some complement $\hat{Z}$ of $Z$, and
$p$ is a prefix of some complement of $X$;

\item[(4')] $u = X Y u'$, $v = \ol{X} \ol{Y} \ol{Z} v''$ for some uniquely
determined proper
complement $\ol{XYZ}$ of $XYZ$, but
$Z$ is not a prefix of $u'$ and $u' \equiv Z v''$, and also
$p$ is a prefix of some complement of $X$;

\item[(5')] $u = X Y Z u''$, $v = \ol{X} \ol{Y} v'$ for some uniquely
determined proper
complement $\ol{XYZ}$ of $X$, 
but $\ol{Z}$ is not a prefix of $v'$ and $\ol{Z} u'' \equiv v'$,
and also $p$ is a prefix of some complement of $X$;

\item[(6')] $u = X Y u'$, $v = \ol{X} \ol{Y} v'$ for some uniquely
determined proper
complement $\ol{XYZ}$ of $XYZ$, $Z$ is not
a prefix of $u'$ and $\ol{Z}$ is not a prefix of $v'$, but
$Z = z_1 z$, $\ol{Z} = z_2 z$, $u' = z_1 u''$, $v' = z_2 v''$ where
$u'' \equiv v''$, $z$ is the maximal common suffix of $Z$ and $\ol{Z}$,
$z$ in non-empty, $z$ is a possible prefix of $u''$, and
also $p$ is a prefix of some complement of $X$.
\end{itemize}
\end{lemma}
\begin{proof}
Mutual exclusivity of the six conditions is proved exactly as for
Lemma~\ref{lemma_eq}. Suppose now that one of the six conditions above applies. Each condition
clearly implies the corresponding condition from Lemma~\ref{lemma_eq},
so we deduce immediately that $u \equiv v$. We must show, using the fact
that $p$ is a prefix of a complement of $X$, that $p$ is a possible prefix
of $u$, or equivalently of $v$.

In case (1'), $p$ is clearly a possible prefix of $u = XYZu''$, and cases
(3'), (4') and (5') are entirely similar.
In case (2'), if $p$ is a prefix of $X$ then
it is already a prefix of $u$, while if $p$ is a prefix of a proper
complement $\ol{X}$ of $X$ and $Z$ is a
possible prefix of $u'$, say $u' \equiv Z w$, then
$$u \ = \ XYu' \ \equiv \ XYZw \ \equiv \ \ol{XYZ} w$$
where the latter has $p$ as a possible prefix.
Finally, in case (6') we know that $z$ is a possible prefix of $u''$, say
$u'' \equiv z x$, so we have
$$u = XYu' = XYz_1u'' = XYz_1zx = XYZx$$
and it is again clear that $p$ is a possible prefix of $u$.

Conversely, suppose $u \equiv v$ and $p$ is a possible prefix of $u$. Then
exactly one of the six conditions in Lemma~\ref{lemma_eq} applies. By
Lemma~\ref{lemma_overlap}, every word equivalent to $u$ begins with a
complement of $XY$, so $p$ must be a prefix of a word beginning with
some complement $\hat{X} \hat{Y}$. Since $\hat{X}$ is the maximal piece prefix of
$\hat{X} \hat{Y} \hat{Z}$ and $\hat{Y}$ is non-empty, it
follows that $p$ is a prefix of $\hat{X}$. If any but condition (2)
of Lemma~\ref{lemma_eq} is satisfied, this suffices to show
that the corresponding condition from the statement of
Lemma~\ref{lemma_eqandprefix} holds.

If condition (2) from Lemma~\ref{lemma_eq} applies, we must show
additionally that either $p$ is a prefix of $X$, or that $Z$ is a
possible prefix of $u'$. Suppose $p$ is not
a prefix of $X$. Then by the above, $p$ is a prefix of some complement
$\hat{X}$. It follows from Lemma~\ref{lemma_overlap}, that the
only way the prefix $XY$ of the word $u$ can be changed using the defining
relations is by application of
a relation of the form $(XYZ, \ol{XYZ})$. In order for this to happen, one must
clearly be able to rewrite $u = XYu'$ to a word of the form $XYZ w$;
consider the shortest possible rewrite sequence which achieves this.
By Lemma~\ref{lemma_overlap}, no term in the sequence except for the last
term will contain a relation word overlapping the initial $XY$. It follows
that the same rewriting steps rewrite $u'$ to $Zw$, so that $Z$ is a
possible prefix of $u'$, as required.
\end{proof}

\section{Applications}\label{sec_apps}

The main application presented in \cite{K_smallover1} was for each 
strongly $C(4)$ monoid presentation, a linear time recursive algorithm to decide,
given words $u$, $v$ and a piece $p$, whether $u \equiv v$ and $p$ is
a possible prefix of $u$. In particular, by fixing $p = \epsilon$, we
obtain an algorithm which
solves the word problem for the presentation in linear time.
Figure~1 shows a modified version of the algorithm which works for weakly
$C(4)$ presentations. The proofs of correctness and
termination are essentially the same as those in \cite{K_smallover1}, but
relying on the more general results of Section~\ref{sec_main}. Thus, we
establish the following theorem.

\begin{theorem}\label{thm_lineartime}
For every weakly $C(4)$ finite monoid presentation, there exists a
two-tape Turing machine which solves the corresponding word problem in
time linear in the lengths the input words.
\end{theorem}

\begin{figure}
\begin{codebox}
\Procname{$\proc{WP-Prefix}(u, v, p)$}
\li     \If $u = \epsilon$ or $v = \epsilon$
\li         \Then \If $u = \epsilon$ and $v = \epsilon$ and $p = \epsilon$  \label{li_start_a}
\li             \Then \Return \const{Yes}                \label{li_allepsilon}
\li             \Else \Return \const{No}                 \label{li_someepsilon}
            \End \label{li_end_a}
\li     \ElseIf $u$ does not have the form $XYu'$ with $XY$ a clean overlap prefix
\li     \Then \If $u$ and $v$ begin with different letters \label{li_start_b}
\li         \Then \Return \const{No}                     \label{li_uvdifferentstart}
\li        \ElseIf $p \neq \epsilon$ and $u$ and $p$ begin with
different letters
\li         \Then \Return \const{No}                     \label{li_updifferentstart}
\li         \ElseNoIf
\li       $u \gets u$ with first letter deleted
\li      $v \gets v$ with first letter deleted
\li      \If $p \neq \epsilon$
\li          \Then $p \gets p$ with first letter deleted
         \End
\li      \Return $\proc{WP-Prefix}(u,v,p)$   \label{li_rec_nomop}
\End \label{li_end_b}

\li \ElseNoIf
\li $\kw{let}\  X, Y, u'$ be such that $u = XY u'$ \label{li_start_c}

\li \If $p$ is not a prefix of a complement of $X$
\li \Then \Return \const{No} \label{li_pnotprefix}

\li \ElseIf $v$ does not begin with a complement of $XY$
\li \Then \Return \const{No} \label{li_vstartswrong}

\li \ElseIf $u = XYZ u''$ and $v = \ol{XYZ} v''$ for some complement $\ol{XYZ}$ of $XYZ$
\li    \Then \If $u''$ is $\hat{Z}$-active for some complement $\hat{Z}$ of $Z$
\li       \Then \Return $\proc{WP-Prefix}(\hat{Z} u'', \hat{Z} v'', \epsilon)$ for some such $\hat{Z}$ \label{li_rec_case1b}
\li       \Else \Return $\proc{WP-Prefix}(Z u'', Z v'', \epsilon)$ \label{li_rec_case1a}
       \End

\li \ElseIf $u = XY u'$ and $v = XY v'$
\li     \Then \If $p$ is a prefix of $X$
\li         \Then \Return $\proc{WP-Prefix}(u',v', \epsilon)$ \label{li_rec_case2a}
\li         \Else \Return $\proc{WP-Prefix}(u',v', Z)$ \label{li_rec_case2b}
        \End

\li \ElseIf $u = XY u'$ and $v = \ol{XYZ} v''$ for some complement $\ol{XYZ}$ of $XYZ$
\li     \Then \Return $\proc{WP-Prefix}(u', Z v'', \epsilon)$ \label{li_rec_case4}

\li \ElseIf $u = XYZ u''$ and $v = \ol{XY} v'$ for some complement $\ol{XYZ}$ of $XYZ$
\li     \Then \Return $\proc{WP-Prefix}(\ol{Z} u'', v', \epsilon)$ \label{li_rec_case5}

\li \ElseIf $u = XY u'$ and $v = \ol{XY} v'$ for some complement $\ol{XY}$ of $XY$
\li     \Then \kw{let} $z$ be the maximal common suffix of $Z$ and $\ol{Z}$
\li           \kw{let} $z_1$ be such that $Z = z_1 z$
\li           \kw{let} $z_2$ be such that $\ol{Z} = z_2 z$
\li           \If $u'$ does not begin with $z_1$ or $v'$ does not begin with $z_2$;
\li               \Then \Return \const{NO} \label{li_case6no}
\li               \Else \kw{let} $u''$ be such that $u' := z_1 u''$
\li                     \kw{let} $v''$ be such that $v' := z_2 v''$;
\li                     \Return $\proc{WP-Prefix}(u'', v'', z)$ \label{li_rec_case6} \label{li_end_c}
              \End
        \End
    \End
\end{codebox}
\caption{Algorithm to solve the word problem for a fixed weakly $C(4)$ presentation.} \label{fig_algorithm}

\end{figure}

The algorithms presented \cite[Section~5]{K_smallover1} for finding
the pieces of a presentation and hence testing strong small overlap conditions
may clearly also be used to test the weak variants of those conditions, with
the proviso
that one considers the \textit{set} of relation words in the presentation,
with any duplicates disregarded. In particular, we have:

\begin{corollary}
There is a RAM algorithm which, given as input a finite presentation $\langle \scrA \mid \scrR \rangle$,
decides in time $O(|\scrR|^2)$ whether the presentation is weakly $C(4)$.
\end{corollary}

\begin{theorem}\label{thm_ramuniform}
There is a RAM algorithm which, given as input a weakly $C(4)$ finite presentation
$\langle \scrA \mid \scrR \rangle$ and two words $u, v \in \scrA^*$, decides whether
$u$ and $v$ represent the same element of the semigroup presented in
time
$$O \left( |\scrR|^2 \min(|u|,|v|) \right).$$
\end{theorem}

Just as with the algorithm from \cite{K_smallover1}, the algorithm in
Figure~\ref{fig_algorithm} is essentially a finite state process, and
can be implemented on a $2$-tape prefix-rewriting automaton using a
slight variation on the technique described in the proof of
\cite[Theorem~2]{K_smallover2}. It follows that we have:

\begin{theorem}\label{thm_main}
Let $\langle \scrA \mid \scrR \rangle$ be a finite monoid presentation which is
weakly $C(4)$. Then the relation
$$\lbrace (u, v) \in \scrA^* \times \scrA^* \mid u \equiv v \rbrace$$
is deterministic rational and reverse deterministic rational. Moreover,
one can, starting from the presentation, effectively compute 2-tape
deterministic automata recognising this relation and its reverse.
\end{theorem}

Just as in \cite{K_smallover2}, we obtain as corollaries large number of
other facts about weakly $C(4)$ monoids. For brevity we refrain from explaining
all terms, and instead refer the reader to \cite{K_smallover2} for definitions.
\begin{corollary}
Every monoid admitting a weakly $C(4)$ finite presentation
\begin{itemize}
\item is \textit{rational} (in the sense of Sakarovitch \cite{Sakarovitch87});
\item is word hyperbolic (in the sense of Duncan and Gilman \cite{Duncan04});
\item is asynchronous automatic;
\item has a regular language of linear-time computable normal forms (namely,
the set of words minimal in their equivalence class with respect to the
lexicographical order induced by any total order on the generating set);
\item has a boolean algebra of rational subsets;
\item has uniformly decidable rational subset membership problem; and
\item has rational subsets which coincide with its recognisable subsets.
\end{itemize}
\end{corollary}

\section*{Acknowledgements}

This research was supported by an RCUK Academic Fellowship. The author
thanks Uri Weiss for drawing his attention to the distinction between
weak and strong $C(n)$ conditions, and asking the questions answered by
this paper.

\bibliographystyle{plain}

\def\cprime{$'$} \def\cprime{$'$}

\end{document}